\documentclass[11pt]{amsart}
\usepackage{amssymb}
\usepackage{latexcad}

\newtheorem{theorem}{Theorem}[section]
\theoremstyle{plain}

\newtheorem{corollary}[theorem]{Corollary}

\newtheorem{lemma}[theorem]{Lemma}

\newtheorem{proposition}[theorem]{Proposition}

\numberwithin{equation}{section}

\def\spec#1{\mathfrak{S}(#1)}
\def\per#1{\pi(#1)}
\def\peron#1{N(#1)}
\def\Per#1{\Pi(#1)}
\def\nxt{\rightarrow}
\def\nxo#1{\overset{#1}\nxt}
\def\nxm#1#2{\overset{#1.#2}\nxt}

\title[Missing lengths of rainbow cycles]{Periods in missing lengths of rainbow cycles}

\author{Petr Vojt\v{e}chovsk\'y}

\email{petr@math.du.edu}

\address{Department of Mathematics, University of Denver, 2360 S Gaylord St,
Denver, Colorado 80208, U.S.A.}



\begin{document}

\begin{abstract}
A cycle in an edge-colored graph is said to be rainbow if no two of its edges
have the same color. For a complete, infinite, edge-colored graph $G$, define
\begin{displaymath}
    \spec{G}=\{n\ge 2\;|\;\text{no $n$-cycle of $G$ is rainbow}\}.
\end{displaymath}
Then $\spec{G}$ is a monoid with respect to the operation $n\circ m = n+m-2$,
and thus there is a least positive integer $\per{G}$, the period of $\spec{G}$,
such that $\spec{G}$ contains the arithmetic progression $\{N+k\per{G}\;|\;k\ge
0\}$ for some sufficiently large $N$.

Given that $n\in\spec{G}$, what can be said about $\per{G}$? Alexeev showed
that $\per{G}=1$ when $n\ge 3$ is odd, and conjectured that $\per{G}$ always
divides $4$. We prove Alexeev's conjecture:

Let $p(n)=1$ when $n$ is odd, $p(n)=2$ when $n$ is divisible by four, and
$p(n)=4$ otherwise. If $2<n\in\spec{G}$ then $\per{G}$ is a divisor of $p(n)$.
Moreover, $\spec{G}$ contains the arithmetic progression $\{N+kp(n)\;|\;k\ge
0\}$ for some $N=O(n^2)$. The key observations are: If $2<n=2k\in\spec{G}$ then
$3n-8\in\spec{G}$. If $16\ne n=4k\in\spec{G}$ then $3n-10\in\spec{G}$.

The main result cannot be improved since for every $k>0$ there are $G$, $H$
such that $4k\in\spec{G}$, $\per{G}=2$, and $4k+2\in\spec{H}$, $\per{H}=4$.
\end{abstract}

\maketitle


\section{Introduction}
Let $G$ be a complete, infinite, edge-colored graph. In \cite{BPV}, the
\emph{spectrum} of $G$ was defined as
\begin{displaymath}
    \spec{G}=\{n\ge 2\;|\;\text{no $n$-cycle of $G$ is rainbow}\}.
\end{displaymath}
It is easy to see (cf. \cite[Proposition 3.1]{BPV}) that $\spec{G}$ is a monoid
with respect to the operation $n\circ m = n+m-2$, and thus that $\spec{G}-2$
can be regarded as a submonoid of $(\mathbb N_0,+)$.

It is therefore reasonable to ask if a given submonoid of $(\mathbb N_0,+)$ can
be realized as $\spec{G}-2$ for some $G$. If the submonoid contains $2$, the
answer is ``yes,'' by \cite[Propositions 3.2, 3.3]{BPV}. In particular, there
is $G$ with $\spec{G}=\{2$, $4$, $6$, $8$, $\dots\}$. However, Alexeev \cite{A}
noticed that not every submonoid of $(\mathbb N_0,+)$ can be so realized,
making the situation much more interesting.

Consequently, Alexeev became interested in the related question \emph{``What
can be said about $\spec{G}$, provided $n\in\spec{G}$?}'', and proved:
\begin{enumerate}
\item[(i)] There is $G$ with $\spec{G} = \{2$, $6$, $10$, $14$, $\dots\}$.

\item[(ii)] If $n=2k+1\in\spec{G}$ then $k(2k+1)\in\spec{G}$.

\item[(iii)] If $n=2k+1\in\spec{G}$ then $3n-6\in\spec{G}$.

\item[(iv)] If $n=2k+1\in\spec{G}$ then $m\in\spec{G}$ for every $m\ge
2n^2-13n+23$.
\end{enumerate}
Moreover, he gave a table of computer generated results, among which one can
find: $8\in \spec{G}\Rightarrow 16\in\spec{G}$, $10\in\spec{G}\Rightarrow
22\in\spec{G}$, and $12\in\spec{G}\Rightarrow 26\in\spec{G}$.

All results mentioned up to this point will be used below without reference.

Let us now recall a simple fact about subsemigroups of natural numbers. Given a
subsemigroup $A$ of $(\mathbb N_0,+)$, let $\Per{A} = \{m-n\;|\;m$, $n\in A$,
$n<m\}$, and $\per{A} = \min \Per{A}$.

\begin{proposition}\label{Pr:MyFrob}
Let $A\ne\{0\}$ be a subsemigroup of $(\mathbb N_0,+)$, and let $p\in\Per{A}$.
Then there is $N$ such that $A$ contains the arithmetic progression
$\{N+kp\;|\;k\ge 0\}$. Moreover, $\per{A} = \gcd \Per{A}$.
\end{proposition}
\begin{proof}
Let $p\in\Per{A}$, and fix $m$, $n\in A$ such that $p=m-n$. Pick $s\ge n$. Then
$A$ contains $sn$, $(s-1)n+m = sn+p$, $\dots$, $(s-n)n + nm = sn + np$. Using
$s=n$, we see that $A$ contains $n^2$, $n^2+p$, $\dots$, $n^2+np$. Using
$s=n+p$, we see that $A$ contains $(n+p)n=n^2+np$, $(n+p)n+p = n^2+(n+1)p$,
$\dots$, $(n+p)n + np = n^2+2np$, and so on. Thus $A$ contains the arithmetic
progression $\{n^2+kp\,|\,k\ge 0\}$.

In particular, there are $N_0$, $N_1$ such that $\{N_0+kp;k\ge 0\}\subseteq A$,
$\{N_1+k\per{A};\;k\ge 0\}\subseteq A$. If $p$ does not divide $\per{A}$, there
is $r\in\Per{A}$ such that $r<\per{A}$, a contradiction.
\end{proof}

Let $\peron{A}$ be the least positive integer such that
\begin{displaymath}
    \{\peron{A}+k\per{A}\;|\;k\ge 0\} \subseteq A.
\end{displaymath}
The proof of Proposition \ref{Pr:MyFrob} shows that $\peron{A}\le n^2$ whenever
there are $m$, $n\in A$ such that $m-n=\per{A}$. Here is another way of
estimating $\peron{A}$:

\begin{theorem}[Frobenius coin-exchange problem] Let $A$ be a subsemigroup of
$(\mathbb N_0,+)$ containing relatively prime integers $n$, $m$. Then
$\per{A}=1$ and $\peron{A}\le (n-1)(m-1)$.
\end{theorem}

\begin{corollary}\label{Cr:Frob}
Let $A$ be a subsemigroup of $(\mathbb N_0,+)$ containing $pn$, $pm$, where
$n$, $m$ are relatively prime. Then $\per{A}$ is a divisor of $p$, and there is
$N\le p(n-1)(m-1)$ such that $\{N+kp\;|\;k\ge 0\}\subseteq A$.
\end{corollary}
\begin{proof}
Let $B=\{x/p\;|\; x\in A$ is divisible by $p\}$. Then $B$ is a subsemigroup of
$(\mathbb N_0,+)$, $m\in B$, $n\in B$. By the Frobenius coin-exchange problem,
$\per{B}=1$ and $\peron{B}\le (n-1)(m-1)$. Thus $\per{pB}=p$ and $\peron{pB}\le
p(n-1)(m-1)$. Since $pB$ is a subsemigroup of $A$, we are done.
\end{proof}

Using this terminology, (iv) can be restated roughly as follows: If $\spec{G}$
contains an odd integer $n$ then $\per{\spec{G}} = 1$, and $\peron{\spec{G}} =
O(n^2)$.

Alexeev conjectured in \cite{A} that $\per{\spec{G}}$ is a divisor of $4$ for
every $G$. This would mean that the two constructions yielding $\spec{G}=\{2$,
$4$, $6$, $8$, $\dots\}$ and $\spec{G}=\{2$, $6$, $10$, $14$, $\dots\}$ are
exceptional. We establish his conjecture and more, as described in the
abstract. The asymptotic behavior of $\spec{G}$ is therefore fully understood.

Finally, when $3\in\spec{G}$, let us call $G$ a \emph{Gallai graph}. Note that
Gallai graphs have no rainbow cycles. All finite Gallai graphs can be built
iteratively \cite{Gallai}, and the iterative construction is very useful in
proving results about Gallai graphs. Is there a similar iterative construction
for $G$ with $4\in\spec{G}$, $5\in\spec{G}$, etc? The results obtained here
could shed some light into this question.

\section{The notation and technique}

The edge with vertices $i$, $j$ will be denoted by $(i,j)$, and its color by
$\gamma(i,j)$.

Given $n\in\spec{G}$, how can one go about proving that $m\in\spec{G}$ for some
specific $m>n$? In a typical scenario, we start with a complete graph $G$ on
$m$ vertices $0$, $\dots$, $m-1$ drawn in the usual way (the vertices form a
regular $m$-gon) and color the edges on the perimeter cycle by
$\gamma(i,i+1)=i$, with vertices and edge colors labeled modulo $m$. There is
no a priori restriction on the possible colors of the inner edges of $G$, but
by carefully selecting $n$-cycles in $G$, we might manage to restrict colors on
the inner edges, until, ultimately, we might prove that some inner edge cannot
be colored at all, hence reaching a contradiction. For instance, since the
$n$-cycle $0\nxt 1\nxt 2\nxt\cdots \nxt n-1 \nxt 0$ cannot be rainbow when
$n\in\spec{G}$, we see that $\gamma(0,n-1) \in\{0$, $\dots$, $n-2\}\subset\{0$,
$\dots$, $m-1\}$.

The difficult part in this strategy is the selection of ``good'' $n$-cycles
that lead to a systematic restriction of colors on the inner edges of $G$. Once
a suitable $n$-cycle is found, the argument becomes routine. In this sense, the
drawings of $n$-cycles accompanying our proofs say (almost) everything. Some
$n$-cycles will be used more than once, and that is the reason why we have
labeled their vertices only by letters---the meaning of the letters will be
clarified in every instance the cycle is used.

Most of our results are of the form ``if $n\in\spec{G}$ and $n>c$ then
$\dots$''. It is usually not difficult to verify the conclusion for large
values of $n$, but it takes some effort to pin down the constant $c$. This
suggests a strategy in which a proof is first read with a large enough $n$ in
mind, and once the structure of the proof is understood, the constant $c$ can
be carefully estimated during second pass. In every such proof we point out at
least one step where $n>c$ is needed.


\section{The even case.}

Fix $n\ge 4$. In Lemmas \ref{Lm:n-1}, \ref{Lm:n-5}, \ref{Lm:7}, let $G$ be a
complete, edge-colored graph with $n\in\spec{G}$, containing a rainbow $3n-8$
cycle $0\nxt 1\nxt \cdots\nxt 3n-9\nxt 0$ such that $\gamma(i,i+1)=i$ for every
$i$.

\setlength{\unitlength}{0.8mm}
\begin{figure}[th]
\begin{center}
\input{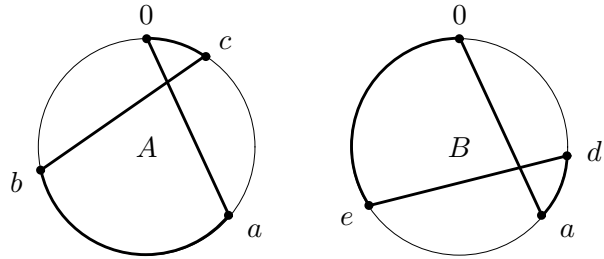}
\end{center}
\caption{Proof of Lemmas \ref{Lm:n-1} and \ref{Lm:n-1x}.} \label{Fg:n-1}
\end{figure}

\begin{lemma}\label{Lm:n-1}
Let $n > 4$. Then we can assume without loss of generality that
\begin{displaymath}
    \gamma(i,i+n-1) \in \{i,i+1\}
\end{displaymath}
for every $i$.
\end{lemma}
\begin{proof}
Consider the $n$-cycle $A$ of Figure \ref{Fg:n-1}, where $a=n-1$, $b=2n-5$,
$c=2$. The edges of $A$ have colors $\gamma(0,n-1)$, $n-1$, $\dots$, $2n-6$,
$\gamma(2,2n-5)$, $1$, $0$. Since $0\nxt 1\nxt \cdots \nxt n-1 \nxt 0$ is an
$n$-cycle and $n\in\spec{G}$, we have $\gamma(0,n-1)\in\{0,\dots,n-2\}$.
Similarly, $\gamma(2,2n-5)\in\{2n-5,\dots,3n-9,0,1\}$. Since $A$ is not
rainbow, at least one of the colors $\gamma(0,n-1)$, $\gamma(2,2n-5)$ occurs
twice on $A$. Thus either $\gamma(0,n-1)\in\{0,1\}$ or
$\gamma(2,2n-5)\in\{0,1\}$. The cycle $A$ is symmetrical with respect to the
line passing through the center and the vertex $1$. The edges $(0,n-1)$ and
$(2,2n-5)$ therefore play a symmetrical role in the construction. Hence,
without loss of generality, $\gamma(0,n-1)\in\{0,1\}$.

Consider the $n$-cycle $B$ of Figure \ref{Fg:n-1}, where $d=n-3$, $e=2n-4$. As
above, we reach the conclusion that either $\gamma(n-3,2n-4)\in\{n-3,n-2\}$, or
$\gamma(0,n-1)\in\{n-3,n-2\}$. But $\gamma(0,n-1)\in\{0,1\}$ and $\{0,1\}\cap
\{n-3,n-2\}=\emptyset$ (we need $n>4$ here). Thus
$\gamma(n-3,2n-4)\in\{n-3,n-2\}$.

Just as we have rotated $A$ clockwise by $n-3$ steps to obtain $B$, we can
rotate $B$ clockwise by $n-3$ steps, etc. Since the integers $3n-8$ and $n-3$
are relatively prime, it follows that $\gamma(i,i+n-1)\in\{i,i+1\}$ for every
$i$.
\end{proof}

\setlength{\unitlength}{0.8mm}
\begin{figure}[th]
\begin{center}
\input{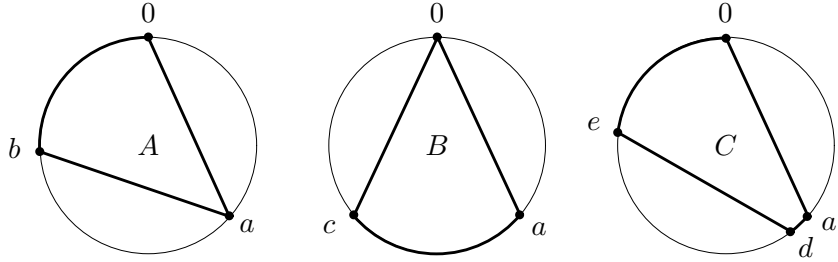}
\end{center}
\caption{Proof of Lemmas \ref{Lm:n-5} and \ref{Lm:n-5x}.} \label{Fg:n-5}
\end{figure}

\begin{lemma}\label{Lm:n-5}
Let $n>5$. Then we can assume without loss of generality that
\begin{displaymath}
    \gamma(i,i+n-1)\in\{i,i+1\},\quad \gamma(i,i+n-5)\in\{i+n-5,i+n-4\}
\end{displaymath}
for every $i$.
\end{lemma}\
\begin{proof}
By Lemma \ref{Lm:n-1}, we can assume that $\gamma(i,i+n-1)\in\{i,i+1\}$ for
every $i$. Consider the $n$-cycle $A$ of Figure \ref{Fg:n-5}, where $a=n-5$ (we
need $n>5$ here), $b=2n-6$. Since $(2n-6) - (n-5) = n-1$, we conclude that
$\gamma(0,n-5)\in\{n-5,n-4\}\cup\{2n-6,\dots,3n-9\}$. Consider the $n$-cycle
$B$ of Figure \ref{Fg:n-5}, where $c=2n-7$. Since $(2n-7)+ (n-1) = 3n-8$, we
conclude that $\gamma(0,n-5)\in\{n-5,\dots,2n-8\}\cup\{2n-7,2n-6\}$. Finally,
consider the $n$-cycle $C$ of Figure \ref{Fg:n-5}, where $d=n-4$ and $e=2n-5$.
Since $(2n-5) - (n-4) = n-1$, we conclude that $\gamma(0,n-5)\in
\{n-5,n-4,n-3\}\cup\{2n-5,\dots,3n-9\}$. Altogether,
$\gamma(0,n-5)\in\{n-5,n-4\}$. As we could have started with any $i$, not
necessarily with $i=0$, we have $\gamma(i,i+n-5)\in\{i+n-5,i+n-4\}$ for every
$i$.
\end{proof}

\setlength{\unitlength}{0.8mm}
\begin{figure}[th]
\begin{center}
\input{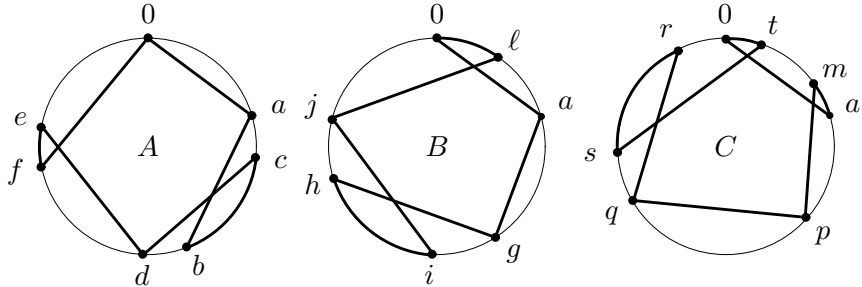}
\end{center}
\caption{Proof of Lemmas \ref{Lm:7} and \ref{Lm:7x}.} \label{Fg:7}
\end{figure}

\begin{lemma}\label{Lm:7}
Let $n>10$. Then we can assume without loss of generality that
\begin{displaymath}
    \gamma(i,i+n-1)\in\{i,i+1\},\,\gamma(i,i+n-5)\in\{i+n-5,i+n-4\},\,
    \gamma(i,i+7)\in\{i,i+1\}
\end{displaymath}
for every $i$.
\end{lemma}
\begin{proof}
By Lemma \ref{Lm:n-5}, we can assume that $\gamma(i,i+n-1)\in\{i,i+1\}$ and
$\gamma(i,i+n-5)\in\{i+n-5,i+n-4\}$ for every $i$ (in fact, we will only need
the second assumption).

Consider an $n$-cycle starting at $0$, containing $1$ forward (clockwise) edge
of length $7$, $4$ forward edges of length $n-5$, and $n-5$ backward
(counterclockwise) edges of length $1$. This is indeed an $n$-cycle, since $7 +
4(n-5) - (n-5) = 3n-8\equiv 0$. The three $n$-cycles of Figure \ref{Fg:7} will
be of this form.

Consider the $n$-cycle $A$ of Figure \ref{Fg:7} with $a=7$, $b=n+2$, $c=9$,
$d=n+4$, $e=2n-1$, $f=2n-3$, and deduce $\gamma(0,7)\in\{0,1\} \cup
\{9,\dots,n+5\} \cup \{2n-3,\dots,2n\}$.

Consider the $n$-cycle $B$ with $g=n+2$, $h=2n-3$, $i=n+7$ (we need $n>10$ to
have $h<i$), $j=2n+2$ and $\ell=5$, and deduce $\gamma(0,7)\in\{0,\dots,6\}
\cup \{n+2,n+3\} \cup \{n+7,\dots,2n-2\} \cup \{2n+2,2n+3\}$. Combined with the
restrictions from cycle $A$, we have $\gamma(0,7)\in\{0,1\}\cup \{n+2,n+3\}
\cup \{2n-3,2n-2\}$.

Finally, consider the $n$-cycle $C$ with $m=5$, $p=n$, $q=2n-5$, $r=3n-10$,
$s=2n$, $t=3$, and deduce $\gamma(0,7)\in \{0,\dots,6\} \cup \{n,n+1\} \cup
\{2n-5,2n-4\} \cup \{2n,\dots,3n-9\}$. Combined with the previous restrictions,
we obtain $\gamma(0,7)\in\{0,1\}$.

Starting at an arbitrary $i$ instead of at $i=0$, we see that
$\gamma(i,i+7)\in\{i,i+1\}$ for every $i$.
\end{proof}

\setlength{\unitlength}{0.7mm}
\begin{figure}[th]
\begin{center}
\input{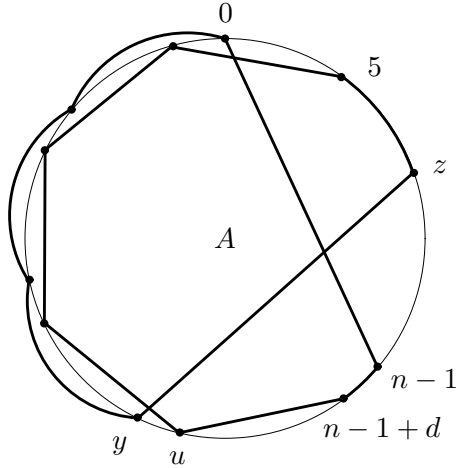}
\end{center}
\caption{Proof of Theorem \ref{Th:Period}.} \label{Fg:TwoRevolutions}
\end{figure}

\begin{theorem}\label{Th:Period}
Let $n\ge 4$ be an even integer. Assume that $G$ is a complete, edge-colored
graph containing no rainbow $n$-cycles. Then $G$ contains no rainbow
$(3n-8)$-cycles.
\end{theorem}
\begin{proof}
There is nothing to show for $n=4$. When $n=6$, we have $10 = 6\circ
6\in\spec{G}$. When $n=8$ (resp.\ $n=10$), Alexeev's computer calculations show
that $16\in\spec{G}$ (resp.\ $22\in\spec{G}$). Therefore, we can take $n\ge
12$. Assume, for a contradiction, that $G$ contains a rainbow $(3n-8)$-cycle
$0\to 1\to\cdots\to 3n-9\to 0$ colored $\gamma(i,i+1)=i$, and restrict all
further considerations to the subgraph of $G$ induced by this cycle. By Lemma
\ref{Lm:7}, we can assume that $\gamma(i,i+7)$, $\gamma(i,i+n-1)\in\{i,i+1\}$
for every $i$.

Note that $2\cdot (n-1) + ((n/2)-2)\cdot 7 + (n/2)\cdot 1 = 6n-16 =
2(3n-8)\equiv 0\pmod{3n-8}$. We will therefore consider an $n$-cycle $A$
containing $2$ edges of length $n-1$, $(n/2)-2$ edges of length $7$, and $n/2$
edges of length $1$, all forward. Such an $n$-cycle makes two clockwise
revolutions modulo $3n-8$.

Assume that $n\ge 16$ (we will deal with $n=12$, $14$ later), and construct $A$
as follows (cf. Figure \ref{Fg:TwoRevolutions}): Start with the edge $(0,n-1)$
followed by the least possible number $d$ of edges of length $1$ so that
$(3n-8) - (n-1+d)\equiv 2\pmod 7$. This is possible, since
\begin{equation}\label{Eq:Ineq1}
    0\le d\le 6\le n/2
\end{equation}
holds whenever $n\ge 12$. Continue the cycle by edges of length $7$ until you
reach $5$. Then add the remaining $(n/2)-d$ edges of length $1$, stopping at
$z$. Note that
\begin{equation}\label{Eq:Ineq2}
    z = 5 + (n/2) - d \le n-1-2
\end{equation}
thanks to $d\ge 0$ and $n\ge 16$. Add an edge of length $n-1$, ending at some
$y = z+n-1$. Note that $y$ is past $n-1+d$ if and only if
\begin{equation}\label{Eq:Ineq3}
    z = 5 + (n/2) - d > d,
\end{equation}
which holds because $2d\le 12 <  5 + 8 \le 5 + n/2$. Hence there is an already
constructed edge $(u,u+7)$ of $A$ such that $u < y\le u+7$. Since we have by
now used all two edges of length $n-1$ and all $n/2$ edges of length $1$, the
cycle $A$ closes itself upon adding the remaining edges of length $7$. But this
shows that $y-u=2$. The conditions $\gamma(i,i+7)$,
$\gamma(i,i+n-1)\in\{i,i+1\}$ then imply that $A$ is rainbow, a contradiction.

When $n=14$, the same construction of $A$ goes through, because $d=5$, and the
needed inequalities \eqref{Eq:Ineq1}--\eqref{Eq:Ineq3} hold. When $n=12$, we
let $A$ be the cycle $0\nxt 11\nxt 12\nxt 13\nxt 14\nxt 15\nxt 16\nxt 23\nxt
2\nxt 9\nxt 20\nxt 27\nxt 0$, for instance.
\end{proof}


\section{The doubly even case}

Fix $n=4k$. In Lemmas \ref{Lm:n-1x}, \ref{Lm:n-5x}, \ref{Lm:7x}, let $G$ be a
complete, edge-colored graph with $n\in\spec{G}$, containing a rainbow $3n-10$
cycle $0\nxt 1\nxt \cdots\nxt 3n-11\nxt 0$ such that $\gamma(i,i+1)=i$ for
every $i$.

\begin{lemma}\label{Lm:n-1x}
Let $n=4k\ge 12$. Then we can assume without loss of generality that
\begin{displaymath}
    \gamma(i,i+n-1)\in\{i,i+1,i+2\}
\end{displaymath}
for every $i$.
\end{lemma}
\begin{proof}
Consider the $n$-cycle $A$ of Figure \ref{Fg:n-1} with $a=n-1$, $b=2n-6$,
$c=3$, and deduce that $\gamma(0,n-1)\in\{0,1,2\}$ or
$\gamma(3,2n-6)\in\{0,1,2\}$. Without loss of generality,
$\gamma(0,n-1)\in\{0,1,2\}$.

The rotated cycle $B$ of Figure \ref{Fg:n-1} with $d=n-4$ and $e=2n-5$ then
shows that either $\gamma(0,n-1)\in\{n-4,n-3,n-2\}$ or
$\gamma(n-4,2n-5)\in\{n-4,n-3,n-2\}$, and hence the latter must be true.

Since $\gcd\{3n-10,n-4\} = 2$, we conclude that
$\gamma(i,i+n-1)\in\{i,i+1,i+2\}$ for every even $i$. If
$\gamma(i,i+n-1)\in\{i,i+1,i+2\}$ holds for at least one odd $i$, it then holds
for every odd $i$, and we are done.

Let us therefore assume, for a contradiction, that
$\gamma(1,1+n-1)\not\in\{1,2,3\}$. Using the $n$-cycle $A$ rotated clockwise by
$1$, we see that $\gamma(2n-5,4)\in\{1,2,3\}$. Proceeding as above but with
counterclockwise rotations by $n-4$, we conclude that
$\gamma(i,i+n-1)\in\{i+n-4,i+n-3,i+n-2\}$ for every odd $i$.

\setlength{\unitlength}{1mm}
\begin{figure}[th]
\begin{center}
\input{one.lp}
\end{center}
\caption{Proof of Lemma \ref{Lm:n-1x}.} \label{Fg:E}
\end{figure}

Note that $n/2$ is even, and consider the $n$-cycle $E$ of Figure \ref{Fg:E}.
(When $n=12$, the vertices $(5/2)n-4$ and $0$ coincide.) Since $0$ is even, $n$
is even, and $(3/2)n-3$ is odd, the edges of $E$ are colored by
$\gamma(0,n-1)\in\{0,1,2\}$, $n-1$, $\gamma(n,2n-1)\in\{n,n+1,n+2\}$,
$(3/2)n-3$, $\dots$, $2n-2$, $\gamma((3/2)n-3, (5/2)n-4)\in\{(5/2)n-7$,
$(5/2)n-6$, $(5/2)n-5\}$, $(5/2)n-4$, $\dots$, $3n-11$. The condition $n\ge 12$
then guarantees that $E$ is rainbow, a contradiction.
\end{proof}

\begin{lemma}\label{Lm:n-5x}
Let $n=4k\ge 12$. Then we can assume without loss of generality that
\begin{displaymath}
    \gamma(i,i+n-1)\in\{i,i+1,i+2\},\quad \gamma(i,i+n-7)\in\{i+n-7,i+n-6,i+n-5\}
\end{displaymath}
for every $i$.
\end{lemma}
\begin{proof}
We can assume that $\gamma(i,i+n-1)\in\{i,i+1,i+2\}$ for every $i$, by Lemma
\ref{Lm:n-1x}. Consider the $n$-cycles $A$, $B$, $C$ of Figure \ref{Fg:n-5}
with $a=n-7$, $b=2n-8$, $c=2n-9$, $d=n-5$, and $e=2n-6$. We get
$\gamma(0,n-7)\in\{n-7$, $n-6$, $n-5\}\cup\{2n-8$, $\dots$, $3n-11\}$ from $A$,
$\gamma(0,7)\in\{n-7$, $\dots$, $2n-7\}$ from $B$, and $\gamma(0,n-7)\in\{n-7$,
$\dots$, $n-3\}\cup\{2n-6$, $\dots$, $3n-11\}$ from $C$. Altogether,
$\gamma(0,n-7)\in\{n-7$, $n-6$, $n-5\}$, and thus $\gamma(i,i+n-7)\in\{i+n-7$,
$i+n-6$, $i+n-5\}$.
\end{proof}

\begin{lemma}\label{Lm:7x}
Let $n=4k\ge 20$. Then we can assume without loss of generality that
\begin{align*}
    &\gamma(i,i+n-1)\in\{i,i+1,i+2\},\quad\gamma(i,i+n-7)\in\{i+n-7,i+n-6,i+n-5\},\\
    &\gamma(i,i+13)\in\{i,i+1,i+2\}
\end{align*}
for every $i$.
\end{lemma}
\begin{proof}
Consider the cycles $A$, $B$, $C$ of Figure \ref{Fg:7} with $a=13$, $b=n+6$,
$c=16$, $d=n+9$, $e=2n+2$, $f=2n-3$, $g=n+6$, $h=2n-1$, $i=n+12$, $j=2n+5$,
$\ell = 8$, $m=8$, $p=n+1$, $q=2n-6$, $r=3n-13$, $s=2n+2$, and $t=5$. Note that
each of the cycles contains $1$ forward edge of length $13$, $4$ forward edges
of length $n-7$, and $n-5$ backward edges of length $1$. The restrictions on
$\gamma = \gamma(0,13)$ obtained from $A$, $B$, and $C$, respectively, are:
$\gamma\in\{0,1,2\} \cup \{16,\dots,n+11\} \cup \{2n-3,\dots, 2n+4\}$,
$\gamma\in \{0,\dots,10\} \cup \{n+6,n+7,n+8\} \cup \{n+12,\dots,2n+1\} \cup
\{2n+5,2n+6,2n+7\}$ (we need $n\ge 20$ to have $2n+7<3n-10$), and
$\gamma\in\{0,\dots,12\} \cup \{n+1,n+2,n+3\} \cup \{2n-6,2n-5,2n-4\} \cup
\{2n+2,\dots,3n-11\}$. It is then easy to see that $\gamma(0,13)\in\{0,1,2\}$,
and thus $\gamma(i,i+13)\in\{i,i+1,i+2\}$ for every $i$.
\end{proof}

We will now imitate the proof of Theorem \ref{Th:Period}. But the situation is
more delicate because we will need cycles with four revolutions, and because
the uncertainty as to the color is larger (3 instead of 2 choices on many
edges). On the other hand, the edges of length $13$ give us a bit more wiggle
room than the edges of length $7$.

\setlength{\unitlength}{1mm}
\begin{figure}[th]
\begin{center}
\input{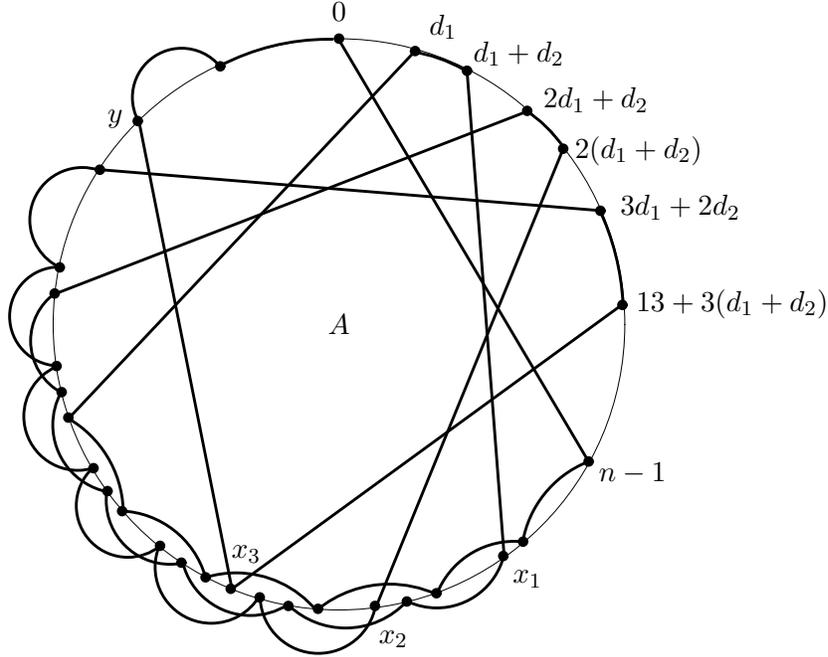}
\end{center}
\vskip - 2cm \caption{Proof of Theorem \ref{Th:Periodx}.}
\label{Fg:FourRevolutions}
\end{figure}

\begin{theorem}\label{Th:Periodx}
Let $n=4k\ne 16$. Assume that $G$ is a complete, edge-colored graph containing
no rainbow $n$-cycles. Then $G$ contains no rainbow $(3n-10)$-cycles.
\end{theorem}
\begin{proof}
When $n=4$, there is nothing to prove since $2\in\spec{G}$. When $n=8$, $8\circ
8 = 14\in\spec{G}$. When $n=12$, Alexeev's computer calculations show that
$26\in\spec{G}$. We can therefore assume that $n=4k\ge 20$. Suppose that $G$
contains a rainbow $(3n-10)$-cycle $0\to 1\to\cdots\to 3n-11\to 0$ colored
$\gamma(i,i+1)=i$, and restrict all further considerations to the subgraph of
$G$ induced by this cycle. By Lemma \ref{Lm:7x}, we can assume that
$\gamma(i,i+n-1)$, $\gamma(i,i+13)\in\{i,i+1,i+2\}$ for every $i$.

Note that $8\cdot (n-1) + (k-2)\cdot 13 + (3k-6)\cdot 1 = 4(3n-10)\equiv
0\pmod{3n-10}$. We will therefore consider an $n$-cycle $A$ containing $8$
edges of length $n-1$, $k-2$ edges of length $13$, and $3k-6$ edges of length
$1$, all forward. Such an $n$-cycle makes four clockwise revolutions modulo
$3n-10$.

Let $r$ be the least integer such that
\begin{displaymath}
    2(4k-1) + 13r \ge 12k-10+3,
\end{displaymath}
i.e.,
\begin{displaymath}
    r = \left\lceil\frac{4k-5}{13}\right\rceil.
\end{displaymath}

We will now describe the cycle (see Figure \ref{Fg:FourRevolutions}) and later
check under which conditions the construction makes sense.

Let $d_2$ be the least nonegative integer such that $2(n-1)+13r - (3n-10) + d_2
\equiv 3 \pmod{13}$, which is equivalent to $d_2\equiv n+8\pmod{13}$. Let $d_1
= 2(n-1)+13r - (3n-10)$. Repeat three times: Add an edge of length $n-1$, $r$
edges of length $13$, an edge of length $n-1$, and $d_2$ edges of length $1$.
(Note the position of $d_1$ in the Figure.) After the three repetitions, add
$13$ edges of length $1$, $2$ edges of length $n-1$, all remaining edges of
length $13$, and all remaining edges of length $1$.

[Enough edges of length 13?] In the construction, we need at least $3r$ edges
of length $13$. We have precisely $k-2$ such edges available. Thus we must have
\begin{equation}\label{Eq:x1}
    3\left\lceil\frac{4k-5}{13}\right\rceil \le k-2.
\end{equation}
Careful analysis of this inequality shows that it holds for all $k\ge 35$.

[Estimating $d_1$ and $d_2$.] By our choice of $r$, $3\le d_1< 3 + 13 = 16$.
Clearly, $0\le d_2\le 12$.

[Enough edges of length $1$?] In the construction, we need at least $13 + 3d_2$
edges of length $1$. We have $3k-6$ such edges. Thus we must have
\begin{equation}\label{Eq:x2}
    13 + 3d_2 \le 3k-6.
\end{equation}
Since $d_2\le 12$, this inequality holds as long as $k\ge 19$.

[Safely below $n-1$ after three rounds?] After the first round ($1$ edge of
length $n-1$, $r$ edges of length $13$, $1$ edge of length $n-1$, $d_2$ edges
of length $1$), we are in position $d_1+d_2$. The second round starts with an
edge of length $n-1$, which brings us to some $x_1 = n - 1 + d_1 + d_2$. By the
definition of $d_2$, we have $x_1>n-1$ and $x_1 - (n-1)\equiv 3\pmod{13}$.
Moreover, upon completing the second round, we will be in position
$2(d_1+d_2)$, and the next edge of length $n-1$ brings us to some $x_2>x_1$
satisfying $x_2 - (n-1) \equiv 6\pmod{13}$. Finally, after the third round we
are in position $3(d_1+d_2)$. We then add $13$ edges of length $1$. In order to
be safely below $n-1$, we demand
\begin{equation}\label{Eq:x3}
    13 + 3(d_1+d_2) \le n - 1 - 3 = 4k-4.
\end{equation}
Since $d_1\le 15$, $d_2\le 12$, this inequality holds whenever $k\ge 25$.

[Far enough after all long edges have been used?] We now continue by adding two
edges of length $n-1$. The first edge moves us to some $x_3>x_2$ satisfying
$x_3 - (n-1)\equiv 9\pmod{13}$. The second edge brings us to $y = 3(d_1+d_2) +
13 + 2(n-1)$. Are we safely past all the edges of length $13$ used so far? The
farthest edge of length $13$ used so far occurred as the last such edge in
round $3$, and terminated at $(n-1) + 13r + 2(d_1+d_2)$. We therefore demand
\begin{equation}\label{Eq:x4}
    (n-1)+13r + 2(d_1+d_2)+3 \le 3(d_1+d_2) + 13 + 2(n-1).
\end{equation}
Since $d_1\ge 3$ and $r <(4k+8)/13$, this inequality always holds.

[Finishing the cycle.] We have used all edges of length $n-1$, completed more
than three revolutions, and there are no edges between our current position $y$
and $0$. The remaining edges of length $1$ and $13$ can therefore be adjoined
in any order, and we are guaranteed to end exactly at $0$.

When all inequalities \eqref{Eq:x1}--\eqref{Eq:x3} hold, the construction
yields a rainbow $n$-cycle, a contradiction.

We are therefore done when $k\ge 35$. Let us have a closer look at $k$ in the
interval $5\le k\le 34$. It is not hard to check that the inequalities
\eqref{Eq:x1}--\eqref{Eq:x4} hold for every $k\in X = \{11$, $20$, $23$, $24$,
$26$, $27$, $29$, $30$, $32$, $33\}$.

For all remaining values $k\in \{5,\dots,34\}\setminus X$, we ran a greedy
algorithm that attempts to construct a $4k$-cycle with $8$ edges of length
$4k-1$, $k-2$ edges of length $13$, and $3k-6$ edges of length $1$, all
forward, so that Lemma \ref{Lm:7x} guarantees that the cycle is rainbow. The
depth-first backtrack algorithm first tries to extend paths by edges of length
$1$, then by edges of length $13$, and finally by edges of length $n-1$. It
succeeds for all $k\in\{5,\dots,34\}\setminus X$, except for $k=22$, $k=25$
(when it was terminated after a few minutes). For instance, it finds $0\nxt
1\nxt 2\nxt 3\nxt 16\nxt 17\nxt 18\nxt 37\nxt 6\nxt 25\nxt 26\nxt 27\nxt 46\nxt
9\nxt 22\nxt 41\nxt 42\nxt 43\nxt 12\nxt 31\nxt 0$ as a valid cycle for $k=5$.

The remaining two cases $k=22$, $k=25$ can be constructed by hand, by
essentially following the general construction. To describe the two cycles, we
use compact notation, in which $\nxo{m}$ means that an edge of length $m$ was
used once, and $\nxm{t}{m}$ means that $t$ edges of length $m$ were used in
succession.

Here is the cycle for $k=22$: $0\nxo{87} 87 \nxm{13}{1} 100 \nxm{6}{13} 178
\nxo{87} 11 \nxm{5}{1} 16 \nxo{87} 103 \nxm{7}{13} 194 \nxo{87} 27 \nxm{5}{1}
32 \nxo{87} 119 \nxm{7}{13} 210 \nxo{87} 43 \nxm{5}{1} 48 \nxm{2}{87} 222
\nxm{32}{1} 0$. And here is the cycle for $k=25$: $0 \nxo{99} 99 \nxm{13}{1}
112 \nxm{7}{13} 203 \nxo{99} 12 \nxm{4}{1} 16 \nxo{99} 115 \nxm{8}{13} 219
\nxo{99} 28 \nxm{4}{1} 32 \nxo{99} 131 \nxm{8}{13} 235 \nxo{99} 44 \nxm{4}{1}
48 \nxm{2}{99} 246 \nxm{44}{1} 0$.
\end{proof}

\section{Main result}

\begin{theorem}\label{Th:Main}
Let $G$ be an infinite, complete, edge-colored graph with $n\in\spec{G}$.
\begin{enumerate}
\item[(i)] If $n\ge 3$ is odd, there is $N\le 2n^2-13n+23$ such that
    $\{N+k;\;k\ge 0\}\subseteq \spec{G}$.

\item[(ii)] If $n=4k+2\ge 6$, there is $N\le (9/4)n^2-18n+37$ such that
    $\{N+4k;\;k\ge 0\}\subseteq \spec{G}$.

\item[(iii)] If $n=4k$, there is $N\le (9/2)n^2-39n+86$ such that
    $\{N+2k;\;k\ge 0\}\subseteq \spec{G}$.
\end{enumerate}
\end{theorem}
\begin{proof}
Part (i) is proved in \cite{A}.

Assume that $n=4k+2\ge 6$. Then $\spec{G}$ contains $n\circ n\circ n=3n-4$, and
also $3n-8$, by Theorem \ref{Th:Period}. Hence $A=\spec{G}-2$ contains $3n-6 =
4(3k)$ and $3n-10 = 4(3k-1)$. Since $3k$, $3k-1$ are relatively prime,
Corollary \ref{Cr:Frob} implies that $A$ contains $\{N+4k;\;k\ge 0\}$ for some
$N\le 4(3k-2)(3k-1)$. Hence $\spec{G}$ contains $\{N+4k;\;k\ge 0\}$ for some
$N\le 4(3k-2)(3k-1)+2 = (9/4)n^2 - 18n + 37$. This proves (ii).

Assume that $n=4k$. If $n=4$ then $\spec{G}$ contains all positive even
integers, and $9/2\cdot n^2 - 39n + 86 = 2$, so (ii) holds. Assume that
$4<n=4k\ne 16$. Then $\spec{G}$ contains $3n-8$ by Theorem \ref{Th:Period}, and
also $3n-10$ by Theorem \ref{Th:Periodx}. Thus $A=\spec{G}-2$ contains $3n-10 =
2(6k-5)$ and $3n-12 = 2(6k-6)$. Since $6k-5$, $6k-6$ are relatively prime (we
need $n>4$ here, else $6k-6=0$), Corollary \ref{Cr:Frob} implies that $A$
contains $\{N+2k;\;k\ge 0\}$ for some $N\le 2(6k-7)(6k-6)$. Hence $\spec{G}$
contains $\{N+2k;\;k\ge 0\}$ for some $N\le 2(6k-7)(6k-6) + 2 = (9/2)n^2 -
39n+86$.

Finally, assume that $n=16$. We have $16$, $3\cdot 16-8 = 40\in\spec{G}$, by
Theorem \ref{Th:Period}, and also $3\cdot 40-10 = 110\in\spec{G}$, by Theorem
\ref{Th:Periodx}. Then $A=\spec{G}-2$ contains $14$, $38$ and $108$, and a
check by computer shows that it contains $\{216+2k;\;k\ge 0\}$. Hence
$\spec{G}$ contains $\{218+2k;\;k\ge 0\}$. Since $218\le 9/2 \cdot 16^2 -
39\cdot 16 + 86 = 614$, we are through.
\end{proof}

\section{Comments and suggestions for future research}

\subsection{Comments}

1) The notion of $\spec{G}$ makes sense for finite graphs, too, but then we
have $\per{\spec{G}}=1$, since there are no $n$-cycles for sufficiently large
$n$.

2) The results on $\per{\spec{G}}$ for $n=4k$, $n=4k+2$ cannot be improved in
general, as witnessed by any graph $G$ with $\spec{G}=\{2$, $4$, $6$, $\dots\}$
(see \cite{BPV} for an example), and by any graph $H$ with $\spec{H} = \{2$,
$6$, $10$, $\dots\}$ (see \cite{A} for an example).

3) Even if the bound in Lemma \ref{Lm:7x} could be improved to $n\ge 16$,
Theorem \ref{Th:Periodx} would not go through for $n=16$ with the current
proof, since there is no ``valid''  $16$-cycle with $8$ edges of length $15$,
$2$ edges of length $13$, and $6$ edges of length $1$, all forward (by a
computer search).

4) The greedy algorithm used in the proof of Theorem \ref{Th:Periodx} is
available at
\begin{displaymath}
    \text{http://www.math.du.edu/\~{}petr}.
\end{displaymath}

\subsection{Suggestions for future research}

Let $n\in\spec{G}$.

1) Improve the bound on $\peron{\spec{G}}$. Can you do better than
$\peron{\spec{G}} = O(n^2)$? We are not aware of any examples in which
$\peron{\spec{G}}$ would get close to $n^2$.

2) For a more ambitious project, describe $\spec{G}$ between $n$ and
$\peron{\spec{G}}$.

3) Extend Theorem \ref{Th:Periodx} to cover the case $n=16$, if possible.

4) Find an iterative construction for all finite, complete, edge-colored graphs
with $4\in\spec{G}$.

\section{Acknowledgement}

I would like to thank Lars-Daniel \"{O}hman for useful conversations at the
2006 Horizon of Combinatorics conference. I would also like to thank the two
anonymous referees for their exceptionally detailed and insightful
reports---the paper has benefited greatly from their comments.


\end{document}